\documentclass[11pt]{amsart}
\usepackage{graphicx}

\usepackage{amssymb}
\usepackage[all]{xy}

\theoremstyle{plain}
\newtheorem{theorem}{Theorem}[section]
\newtheorem{proposition}[theorem]{Proposition}
\newtheorem{lemma}[theorem]{Lemma}
\newtheorem{corollary}[theorem]{Corollary}

\theoremstyle{definition}
\newtheorem{definition}[theorem]{Definition}
\newtheorem{notation}[theorem]{Notation}

\theoremstyle{remark}
\newtheorem{example}[theorem]{Example}
\newtheorem{remark}[theorem]{Remark}
\newtheorem{warning}[theorem]{Warning}

\newcommand{\Z}{\mathbb{Z}}

\newcommand{\AAA}{\mathbf{\mathcal{A}}}
\newcommand{\CC}{\mathbf{\mathcal{C}}}
\newcommand{\DD}{\mathbf{\mathcal{D}}}

\newcommand{\al}{\alpha}
\newcommand{\be}{\beta}
\newcommand{\de}{\delta}
\newcommand{\De}{\Delta}
\newcommand{\ep}{\epsilon}
\newcommand{\Ga}{\Gamma}
\newcommand{\ga}{\gamma}
\newcommand{\io}{\iota}
\newcommand{\ka}{\kappa}
\newcommand{\La}{\Lambda}
\newcommand{\la}{\lambda}
\newcommand{\Om}{\Omega}
\newcommand{\phy}{\varphi}
\newcommand{\Si}{\Sigma}
\newcommand{\te}{\theta}

\newcommand{\inj}{\hookrightarrow}
\newcommand{\ral}{\xrightarrow} 
\newcommand{\Ra}{\Rightarrow}
\newcommand{\surj}{\twoheadrightarrow}
\newcommand{\tild}{\widetilde}

\newcommand{\dfn}{:=}
 
\newcommand{\lan}{\left\langle}
\newcommand{\op}{\oplus}
\newcommand{\Op}{\bigoplus}
\newcommand{\ot}{\otimes}
\newcommand{\ol}{\overline}
\newcommand{\ran}{\right\rangle}

\newcommand{\ul}{\underline}
\newcommand{\We}{\bigvee}
\newcommand{\x}{\times}

\newcommand{\Ab}{\mathbf{Ab}}
\newcommand{\Ho}{\mathbf{Ho}}
\newcommand{\Mod}{\mathbf{Mod}}
\newcommand{\Modff}{\mathbf{Mod}^{\mathrm{ff}}}
\newcommand{\PI}{\pmb{\Pi}}
\newcommand{\PIst}{\PI^{\mathrm{st}}}
\newcommand{\PiAlg}{\PI\mathbf{Alg}}
\newcommand{\Set}{\mathbf{Set}}
\newcommand{\Sp}{\mathbf{Sp}}
\newcommand{\Topp}{\mathbf{Top}}
\newcommand{\TT}{\mathbf{T}}

\newcommand{\abs}[1]{\lvert #1 \rvert}

\DeclareMathOperator{\coker}{coker}
\DeclareMathOperator{\Ext}{Ext}
\DeclareMathOperator{\Fun}{Fun}
\DeclareMathOperator{\Hom}{Hom}
\DeclareMathOperator{\im}{im}
\DeclareMathOperator{\Lan}{Lan}
\DeclareMathOperator{\Tor}{Tor}

\newcommand{\col}{\mathrm{col}}
\newcommand{\hi}{\mathrm{hi}}
\newcommand{\id}{\mathrm{id}}
\newcommand{\inc}{\mathrm{inc}}
\newcommand{\lo}{\mathrm{lo}}
\newcommand{\model}{\mathrm{model}}
\newcommand{\opp}{\mathrm{op}}

\newcommand{\Def}{\textbf}

\begin{document}

\title[The realizability of operations on homotopy groups]{The realizability of operations on homotopy groups concentrated in two degrees}

\author{Hans-Joachim Baues}
\address{Max-Planck-Institut f\"ur Mathematik, Vivatsgasse 7, 53111 Bonn, Germany}
\email{baues@mpim-bonn.mpg.de}

\author{Martin Frankland}
\address{Department of Mathematics, University of Western Ontario, Middlesex College, London, ON  N6A 5B7, Canada}
\email{mfrankla@uwo.ca}

\begin{abstract}
The homotopy groups of a space are endowed with homotopy operations which define the $\Pi$-algebra of the space. An Eilenberg-MacLane space is the realization of a $\Pi$-algebra concentrated in one degree. In this paper, we provide necessary and sufficient conditions for the realizability of a $\Pi$-algebra concentrated in two degrees. We then specialize to the stable case, and list infinite families of such $\Pi$-algebras that are not realizable.
\end{abstract}



\thanks{The second author was supported in part by a Postdoctoral Research Fellowship from the Fonds Qu\'eb\'ecois de la Recherche sur la Nature et les Technologies (FQRNT). He would like to thank the Max-Planck-Institut f\"ur Mathematik Bonn for its generous hospitality, as well as Katja Hutschenreuter, Markus Szymik, Haynes Miller, Charles Rezk, Paul Goerss, Angelica Osorno, Doug Ravenel, and Mark Behrens for fruitful conversations. The authors also thank the referee for useful comments.}

\date{\today}

\maketitle

\section{Realization problem for homotopy operations}

The homotopy groups $\pi_* X$ of a pointed space $X$ are not merely a list of groups, but carry the additional structure of an action of the (primary) homotopy operations, which are natural transformations
\[
\pi_{n_1} X \x \pi_{n_2} X \x \ldots \x \pi_{n_j} X \to \pi_n X.
\]
These include for example Whitehead products $\pi_p X \x \pi_q X \to \pi_{p+q-1} X$, as well as precomposition operations $\al^* \colon \pi_m X \to \pi_n X$ induced by any map $\al \colon S^n \to S^m$, defined by $\al^*(x) = x \circ \al$. By the Yoneda lemma, $j$-ary homotopy operations are parametrized by homotopy classes of pointed maps
\[
S^n \to S^{n_1} \vee S^{n_2} \vee \ldots \vee S^{n_j}.
\]
This information is encoded in a category as follows.

\begin{definition}
Let $\Topp_*$ denote the category of pointed topological spaces. Let $\PI$ denote the full subcategory of the homotopy category $\Ho \Topp_*$ consisting of finite wedges of spheres $\vee S^{n_i}$, $n_i \geq 1$. Note that the empty wedge (a point) is allowed.

A \Def{$\Pi$-algebra} is a product-preserving functor $\PI^{\opp} \to \Set$, in other words, a contravariant functor $\PI \to \Set$ which sends wedges to products. Let $\PiAlg$ denote the category of $\Pi$-algebras, where morphisms are natural transformations.
\end{definition}

The prototypical example is the homotopy $\Pi$-algebra $[-,X]$ of a pointed space $X$, which is the functor represented by $X$ in the homotopy category. One can view this data as the graded group $\pi_* X$, with $\pi_n X = [S^n,X]$, endowed with the structure of primary homotopy operations. Likewise, given any $\Pi$-algebra $\ul{A}$, the group $\ul{A}(S^n)$ will be denoted $A_n$. Taking the homotopy groups $\pi_* X$ defines a functor $\pi_* \colon \Ho \Topp_* \to \PiAlg$ sending $X$ to its homotopy $\Pi$-algebra.

\begin{definition}
A $\Pi$-algebra $\ul{A}$ is called \Def{realizable} if there is a space $X$ together with an isomorphism $\ul{A} \simeq \pi_* X$ of $\Pi$-algebras. Such a space $X$ is called a \Def{realization} of $\ul{A}$.
\end{definition}

\begin{example}
A $\Pi$-algebra concentrated in a single degree $n$ is the same as a group $A_n$, which is abelian if $n \geq 2$. All such $\Pi$-algebras are realizable (uniquely up to weak equivalence), and the Eilenberg-MacLane space $K(A_n,n)$ is a realization of this $\Pi$-algebra.
\end{example}

In general, one has the following \textbf{realization problem}: Given a $\Pi$-algebra $\ul{A}$, is $\ul{A}$ rea\-li\-zable by a space? Here, one must realize not only the homotopy groups, but also the prescribed homotopy operations.

\subsection*{Background on the problem}

One has the following classic example due to Quillen.

\begin{example}
Let $\ul{A}$ be a simply-connected rational $\Pi$-algebra, i.e., satisfying $A_1 = 0$ and $A_n$ is a rational vector space. Then $\ul{A}$ is realizable. In fact, the category of such $\Pi$-algebras is equivalent to the category of reduced graded Lie algebras, and each such Lie algebra is the Samelson product Lie algebra of a space \cite[Theorem I]{Quillen69}.
\end{example}

\begin{example} \label{NonSimply}
A $\Pi$-algebra concentrated in degrees $1$ and $n$ consists of a group $A_1$ and an $A_1$-module $A_n$, and can be realized by a generalized Eilenberg-MacLane space \cite{Whitehead49}. Moreover, the moduli space of realizations is described in \cite[Theorem 3.4, Corollary 3.5]{Frankland11}.
\end{example}

\begin{example} \label{k1}
A $\Pi$-algebra concentrated in two \emph{consecutive} degrees $n, n+1$ (with $n \geq 2$) consists of two abelian groups $A_n$ and $A_{n+1}$ together with a homomorphism $\Ga_n^1 (A_n) \to A_{n+1}$, where the functor $\Ga_n^1$ is given by
\[
\Ga_n^1(A_n) = 
\begin{cases}
\Ga(A_n) &\text{for } n=2 \\
A_n \ot \Z/2 &\text{for } n \geq 3
\end{cases}
\]
where $\Ga$ denotes Whitehead's quadratic functor. The structure map $\Ga_n^1(A_n) \to A_{n+1}$ corresponds to precomposition $\eta^* \colon A_n \to A_{n+1}$ by the Hopf map $\eta \colon S^{n+1} \to S^n$. More precisely, $\eta^* \colon A_n \to A_{n+1}$ is a quadratic map when $n=2$ (resp. a linear map of order $2$ when $n \geq 3$), and therefore corresponds by adjunction to a map of abelian groups $\Ga_n^1(A_n) \to A_{n+1}$.

All such $\Pi$-algebras are realizable. This follows from J.H.C. Whitehead's homotopy classification of simply connected $4$-dimensional CW-complexes in terms of the certain exact sequence \cite{Whitehead50}; see also \cite[Theorem 3.3 (A)]{Baues00}. Moreover, the moduli space of realizations is described in \cite[Theorem 5.1]{Frankland11}.
\end{example}

\begin{example}
A $\Pi$-algebra concentrated in a stable range can be identified with a module over the stable homotopy ring $\pi_*^S$, i.e., the homotopy groups of the sphere spectrum; see Section \ref{sec:Stable}. Our results provide examples of such modules that are not realizable (by a space or, equivalently, by a spectrum).
\end{example}

For more background on $\Pi$-algebras, see for example \cite[\S 4]{Stover90} \cite[\S 3.1]{Blanc90} \cite[\S 2]{Blanc93} \cite[\S 2]{Dwyer94} \cite[\S 4]{Blanc04}. For literature on the realization problem for $\Pi$-algebras and some generalizations, see for example \cite{Blanc95} \cite{Blanc99} \cite{Blanc04} \cite{Blanc06}.

\subsection*{Main results and organization}

In Section \ref{sec:2stage}, we describe $\Pi$-algebras concentrated in two degrees in terms of homotopy groups of spheres (Proposition \ref{GammaAdditive}). Section \ref{sec:Metastable} is devoted to the metastable case in degrees $n$ and $2n-1$ (Proposition \ref{GammaQuadratic}).

Section \ref{sec:Criterion} explains the main result of this paper, which solves the realization problem for $\Pi$-algebras concentrated in two degrees. Theorem \ref{Criterion} provides a necessary and sufficient condition for such a $\Pi$-algebra to be realizable, in terms of homology of Eilenberg-MacLane spaces.

Section \ref{sec:Stable} specializes to the stable case. In Section \ref{sec:NonRealiz}, we provide infinite families of non-realizable examples, using elements in the image of the $J$-homomorphism (Propositions \ref{Alpha} and \ref{DividedAlpha}). Section \ref{sec:Proofs} contains proofs and technical material that would have otherwise cluttered the exposition.

\subsection*{Notations and conventions}

All tensor products will be over $\Z$ unless otherwise stated, so that we write $\ot \dfn \ot_{\Z}$.

A $\Pi$-algebra $\ul{A}$ is called \Def{$m$-truncated} if it satisfies $A_i = 0$ for $i > m$ and \Def{$m$-connected} if it satisfies $A_i = 0$ for $i \leq m$. We will be working with $\Pi$-algebras concentrated in degrees $n, n+1, \ldots, n+k$ for integers $n \geq 2$ and $k \geq 0$, in other words, $(n-1)$-connected $(n+k)$-truncated $\Pi$-algebras. We adopt the following notation, which suggests ``starting in degree $n$ at the bottom and going up $k$ degrees'':
\begin{itemize}
 \item $\PiAlg_n$ is the full subcategory of $\PiAlg$ consisting of $(n-1)$-connected $\Pi$-algebras.
 \item $\PiAlg_n^k$ is the full subcategory of $\PiAlg$ consisting of $\Pi$-algebras concentrated in degrees $n$ to $n+k$.
\end{itemize}
We use a similar convention for categories of spheres of certain dimensions:
\begin{itemize}
 \item $\PI_n$ is the full subcategory of $\PI$ consisting of wedges of spheres of dimensions at least $n$.
 \item $\PI_n^k$ is the full subcategory of $\PI$ consisting of wedges of spheres of dimensions from $n$ to $n+k$.
\end{itemize}
We will use analogous notations for the stable picture in Section \ref{sec:Proofs}.

\section{Homotopy operation functors} \label{sec:2stage}

In this section, we first recall the machinery of \cite[\S 1]{Baues00} encoding homotopy operations inductively, one degree at a time. Then, we specialize to $\Pi$-algebras concentrated in two degrees.

\subsection*{Truncated $\Pi$-algebras} 

The Postnikov truncation functor $P_{n+k-1} \colon \PiAlg_n^k \to \PiAlg_n^{k-1}$ admits a left adjoint $L$. As in \cite[Definition 1.5]{Baues00}, consider the \Def{homotopy operation functor} $\Ga_n^k \colon \PiAlg_n^{k-1} \to \Ab$ defined as  the composite
\[
\xymatrix{
\PiAlg_n^{k-1}  \ar@/^2pc/[rr]^-{\Ga_n^k} \ar[r]^-L & \PiAlg_n^{k} \ar[r]^-{\pi_{n+k}} & \Ab \\
}
\]
where $\pi_{n+k} \colon  \PiAlg_n^{k} \to \Ab$ is evaluation on the sphere $S^{n+k}$, which extracts from a $\Pi$-algebra $\ul{A}$ the abelian group $A_{n+k} = \ul{A}(S^{n+k})$. Using these functors, $\PiAlg_n^k$ can be described as an iterated comma category
\[
\PiAlg_n^k \cong \Ga_n^k \Ab
\]
as in \cite[Proposition 1.6]{Baues00}. Note that the inductive process starts with $\PiAlg_n^0 \cong \Ab$ (assuming $n \geq 2$). Let us recall some terminology and notation for comma categories \cite[Definition 1.1]{Baues99} \cite[\S 1.5]{Baues00}.

\begin{definition} \label{CommaCat}
Let $\CC$ be a category and let $\Ga \colon \CC \to \AAA$ be a functor. Then we obtain the category $\Ga \AAA$ as follows. An object is a triple $(X,A,\eta)$ where $X$ is an object of $\CC$ and $\eta \colon \Ga X \to A$ is a morphism in $\AAA$. A morphism $(X,A,\eta) \to (Y,B,\la)$ in $\Ga \AAA$ is a pair $(f,g)$ where $f \colon X \to Y$ is a morphism in $\CC$ such that the diagram
\[
\xymatrix{
\Ga X \ar[r]^{\Ga f} \ar[d]^{\eta} & \Ga Y \ar[d]^{\la} \\
A \ar[r]^g & B \\
}
\]
commutes in $\AAA$. We call $\Ga \AAA$ the \Def{comma category} of $\Ga$. An object $(X,A,\eta)$ of $\Ga \AAA$ is also denoted by $\eta$.
\end{definition}

Comma categories are also described in \cite[\S 2.6]{Maclane98}, where our $\Ga \AAA$ is denoted $(\Ga \downarrow 1_{\AAA})$ or $(\Ga \downarrow \AAA)$. We will use the following facts about comma categories, whose proofs are straightforward.

\begin{lemma} \label{FunctorComma}
Functors $F,G \colon \CC \to \DD$ are isomorphic if and only if the comma categories $F \DD, G \DD$ are equivalent as categories over $\CC \x \DD$. Here the projection $F \DD \to \CC \x \DD$ sends an object $(X,A,\eta)$ to $(X,A)$.
\end{lemma}

\begin{lemma} \label{CommaAdditive}
Let $\CC, \DD$ be additive categories and $F \colon \CC \to \DD$ a functor. Then the comma category $F \DD$ is additive if and only if $F$ is an additive functor.
\end{lemma}

\subsection*{$\Pi$-algebras concentrated in two degrees} 

Let $\PiAlg(n,n+k)$ be the full subcategory of $\PiAlg$ consisting of $\Pi$-algebras concentrated in degrees $n$ and $n+k$ for some $n,k \geq 1$; these are sometimes called \Def{$2$-stage} $\Pi$-algebras. In light of Example \ref{NonSimply}, we will assume $n \geq 2$. The category $\PiAlg(n,n+k)$ can be described as a comma category as follows.

\begin{proposition} \label{Gamma2}
Let $n \geq 2$. There is a unique functor (up to natural isomorphism) $\tild{\Ga}_n^k \colon \Ab \to \Ab$ yielding an isomorphism
\[
\PiAlg(n,n+k) \cong \tild{\Ga}_n^k \Ab
\]
of categories over $\Ab \x \Ab$.

For example, in the case $k=1$, the functor $\tild{\Ga}_n^1 = \Ga_n^1$ is described in Example \ref{k1}.
\end{proposition}

\begin{proof}
Uniqueness follows from \ref{FunctorComma}. For existence, take
\[
\tild{\Ga}_n^k(A_n) = \Ga_n^k(A_n, 0, \ldots, 0)
\]
where $(A_n, 0, \ldots, 0)$ denotes the (unique) object $\ul{A}$ of $\PiAlg_n^{k-1}$ with $A_{n+1} = 0$, $\ldots$, $A_{n+k-1} = 0$. In other words, $\tild{\Ga}_n^k$ is the restriction of $\Ga_n^k \colon \PiAlg_n^{k-1} \to \Ab$ to the full subcategory $\Ab \cong \PiAlg_n^0 \inj \PiAlg_n^{k-1}$. The full subcategory $\PiAlg(n,n+k)$ of $\PiAlg_n^k$ is isomorphic to the comma category of $\Ga_n^k$ restricted to objects of the form $(A_n,0, \ldots, 0)$, which is precisely the functor $\tild{\Ga}_n^k$.
\end{proof}

In particular, the equality $\tild{\Ga}_n^k = 0$ holds if and only if the projection $\PiAlg(n,n+k) \ral{\cong} \Ab \x \Ab$ is an isomorphism of categories, that is, the $\Pi$-algebra structure concentrated in degrees $n$ and $n+k$ is trivial. The corresponding $\Pi$-algebras $(A_n, A_{n+k})$ are clearly realizable, for example by a product of Eilenberg-MacLane spaces $K(A_n,n) \x K(A_{n+k},n+k)$.

\begin{remark}
By \ref{CommaAdditive} and \ref{Gamma2}, the category $\PiAlg(n,n+k)$ is additive if and only if the functor $\tild{\Ga}_n^k$ is additive. This certainly happens in the stable range, but not always (e.g. $k=2,n=3$ as in Example \ref{k2}). In fact, we will see shortly that it happens often; see Proposition \ref{GammaAdditive}.
\end{remark}

\begin{example} \label{k2}
Taking $k=2$, the formula for $\Ga_n^2$ in \cite[1.10]{Baues00} yields
\[
\tild{\Ga}_n^2(A_n) = 
\begin{cases}
0 &\text{for } n=2 \\
\La^2 (A_3) &\text{for } n=3 \\
0 &\text{for } n \geq 4
\end{cases}
\]
where $\La^2(A) := A \ot A / (a \ot a \sim 0)$ denotes the exterior square. Note that the map $\La^2 (A_3) \to A_5$ encodes the Whitehead product $[-,-] \colon A_3 \ot A_3 \to A_5$.
\end{example}

In a $\Pi$-algebra concentrated in degrees $n$ and $n+k$, any operation that factors through intermediate degrees would automatically vanish. This suggests looking at indecomposable operations, in the following sense.

\begin{definition} \label{Indecomp}
An element $x \in \pi_{n+k}(S^n)$ is called \Def{decomposable} if it admits a factorization
\[
\xymatrix{
S^{n+k} \ar[r]^-{w} & \bigvee S^n \vee \bigvee S^{n_i} \ar[r] & S^n \\ 
}
\]
where the dimensions $n_i$ satisfy $n < n_i < n+k$ and the composite $S^{n+k} \ral{w} \bigvee S^n \vee \bigvee S^{n_i} \surj \bigvee S^n$ of $w$ with the collapse map onto the first summand is null.

This means that $x$ is obtained via primary homotopy operations from elements of lower degree, possibly of degree $n$, but in a way that elements of intermediate degree (between $n$ and $n+k$) are essential. For example, the Whitehead product $[y,\io_n] \in \pi_{i+n-1}(S^n)$ with $y \in \pi_i(S^n)$, $i > n$, is decomposable. However, the Whitehead product $[\io_n,\io_n] \in \pi_{2n-1}(S^n)$ is not considered decomposable, a priori.

Let $Q_{k,n}$ denote the \Def{indecomposables} of $\pi_{n+k}(S^n)$, i.e., the quotient of $\pi_{n+k}(S^n)$ by the subgroup generated by all decomposable elements.

In the stable range $k \leq n-2$, $Q_{k,n} = Q_k^S$ does not depend on $n$. Here $Q_*^S$ denotes the indecomposables of the graded ring $\pi_*^S$ (homotopy groups of the sphere spectrum $S^0$), with respect to the augmentation $\pi_*^S \to \Z$ induced by the Hurewicz map $S^0 \to H\Z$.
\end{definition}

\begin{warning}
The definition of decomposable in \cite[\S 2.2]{Blanc90} \emph{does} include elements obtained via primary operations from elements of degree $n$. In particular, the latter definition makes \emph{every} element $x \in \pi_{n+k}(S^n)$ decomposable, since it is obtained as a precomposition of the identity class, $x = \io_n \circ x = x^* (\io_n)$, as noted in \cite[\S 2.2.2]{Blanc90}. Definition \ref{Indecomp} should be thought of as ``decomposable via intermediate degrees''.
\end{warning}

\begin{remark}
The subgroup generated by all decomposables is in fact generated by compositions of the form $S^{n+k} \to S^m \to S^n$ (with $n < m < n+k$) and $3$-fold iterated Whitehead products of the identity map $\io_n \in \pi_n(S^n)$ of even-dimensional spheres. This follows from the Barcus-Barratt formula and the fact that all $4$-fold iterated Whitehead products of the identity class for spheres vanish \cite[Theorem XI.8.8]{Whitehead78}. See the discussion before \cite[Lemma 3.6]{Blanc93}.
\end{remark}

\begin{proposition} \label{GammaAdditive}
Assuming $k \neq n-1$, we have
\[
\tild{\Ga}_n^k(A_n) = A_n \ot Q_{k,n}.
\]
In particular, in the stable range $k \leq n-2$, we have
\[
\tild{\Ga}_n^k(A_n) = A_n \ot Q^S_k.
\]
\end{proposition}

\begin{proof}
See Section \ref{sec:Proofs}.
\end{proof}

\begin{corollary} \label{TrivialOps}
For all $k$ and $n$ with $k \neq n-1$ such that $Q_{k,n} = 0$ holds, $2$-stage $\Pi$-algebras concentrated in degrees $n$ and $n+k$ have trivial homotopy operations and are thus automatically realizable.
\end{corollary}

\begin{example}
Every $\Pi$-algebra concentrated in degrees $2$ and $2+k$ is realizable. The case $k=1$ is settled in Example \ref{k1}. For the case $k \geq 2$, note that the Hopf map $\eta \colon S^3 \to S^2$ induces an isomorphism $\pi_{2+k} S^3 \ral{\simeq} \pi_{2+k} S^2$. Hence every element in $x \in \pi_{2+k} S^2$ is in fact a decomposable element $\eta \circ x'$ for some $x' \in \pi_{n+k} S^3$. Thus we have $Q_{k,2} = 0$ and the result follows from \ref{TrivialOps}.
\end{example}

As noted in Example \ref{k1}, the realization problem is solved in the affirmative in the case $k=1$. The same is true for the case $k=2$.

\begin{proposition} \label{k2Realizable}
Every $\Pi$-algebra concentrated in degrees $n$ and $n+2$ is realizable.
\end{proposition}

\begin{proof}
In the stable range $n \geq 4$, it follows from \ref{TrivialOps} and $Q_2^S = 0$, because of $\pi_2^S = \Z/2 \lan \eta^2 \ran$. Likewise for $n=2$, it follows from the fact $Q_{2,2} = 0$, obtained from $\pi_4(S^2) = \Z/2 \lan \eta \circ \eta \ran$.

The only case where the $\Pi$-algebra data is non-trivial is $n = 3$, with $\tild{\Ga}_3^2 = \La^2$ as noted in Example \ref{k2}. In that case, the $\Pi$-algebra $\ul{A}$ is realizable if and only if the obstruction $O(\ul{A}) = \eta_2 \circ E_3(\eta_1)$ described in \cite[Theorem 3.3 (B)]{Baues00} vanishes. The map $E_3(\eta_1)$ described in \cite[\S 3.2]{Baues00} factors through $A_4$ and is therefore zero in our case (with $A_4 = 0$).
\end{proof}

\section{Metastable case} \label{sec:Metastable}

The situation is somewhat more complicated for the critical dimension $k = n-1$, which is in the metastable range. Let us recall some terminology and basic facts from \cite{Baues94}.

\begin{definition}
\cite[Definition 2.1]{Baues94} A \Def{quadratic module}
\[
M = \left( M_e \ral{H} M_{ee} \ral{P} M_e \right)
\]
consists of a pair of abelian groups $M_e$ and $M_{ee}$ together with homomorphisms $H$ and $P$ that satisfy $PHP = 2P$ and $HPH = 2H$.

A morphism $f \colon M \to N$ of quadratic modules consists of a pair of homomorphisms $f \colon M_e \to N_e$ and $f \colon M_{ee} \to N_{ee}$ which commute with $H$ and $P$ respectively.

For any quadratic module $M$, one has the involution
\[
T \dfn HP-1 \colon M_{ee} \to M_{ee}
\]
which satisfies $PT = P$, $TH = H$, and $TT = 1$.
\end{definition}

Note that in \cite[Definition 2.1]{Baues94}, quadratic modules are called quadratic $\Z$-modules, because more general ground rings besides $\Z$ are considered.

\begin{example}
\cite[After Remark 9.2]{Baues94} Consider 
\[
\pi_m \{ S^n \} = \left( \pi_m S^n \ral{H} \pi_m S^{2n-1} \ral{P} \pi_m S^n \right)
\]
where $H$ is the Hopf invariant and $P = [\io_n, \io_n]_*$ is induced by the Whitehead square. This data $\pi_m \{ S^n \}$ is a quadratic module. In particular, we have
\begin{align*}
&\pi_3 \{ S^2 \} = \left( \pi_3 S^2 \ral{H} \pi_3 S^3 \ral{P} \pi_3 S^2 \right) = \left( \Z \ral{1} \Z \ral{2} \Z \right) \\
&\pi_5 \{ S^3 \} = \left( \pi_5 S^3 \ral{H} \pi_5 S^5 \ral{P} \pi_5 S^3 \right) = \left( \Z/2 \ral{0} \Z \ral{0} \Z/2 \right).
\end{align*}
\end{example}

\begin{definition}
\cite[Definition 4.1]{Baues94} Given an abelian group $A$ and a quadratic module $M$, their \Def{quadratic tensor product} $A \ot^{q} M$ is the abelian group generated by symbols
\begin{align*}
&a \ot m, \quad a \in A, m \in M_e \\
&[a,b] \ot n, \quad a,b \in A, n \in M_{ee}
\end{align*}
subject to the relations
\begin{align*}
&(a+b) \ot m = a \ot m + b \ot m + [a,b] \ot H(m) \\
&a \ot (m+m') = a \ot m + a \ot m' \\
&[a,a] \ot n = a \ot P(n) \\
&[a,b] \ot n = [b,a] \ot T(n) \\
&[a,b] \ot n \text{ is linear in each variable } a, b, \text{ and } n.
\end{align*}
\end{definition}

\begin{example} \label{WhitQuad}
\cite[Proposition 4.5]{Baues94} Taking the quadratic module
\[
\Z^{\Ga} \dfn \left( \Z \ral{1} \Z \ral{2} \Z \right) \simeq \pi_3 \{ S^2 \},
\]
the quadratic tensor product with any abelian group $A$ is $A \ot^q \Z^{\Ga} \cong \Ga(A)$, Whitehead's universal quadratic functor $\Ga \colon \Ab \to \Ab$ described in \cite{Whitehead50} \cite[\S 2.1]{Baues05}.
\end{example}

Note that the usual tensor product with a given abelian group $M$ defines an additive functor $- \ot M \colon \Ab \to \Ab$. Similarly, the quadratic tensor product with a fixed quadratic module $M$ defines a quadratic functor $- \ot^q M \colon \Ab \to \Ab$ in the following sense.

\begin{definition}
\cite[\S 2]{Baues05} Let $F \colon \Ab \to \Ab$ be a functor satisfying $F(0) = 0$. Recall that $F$ is \Def{additive} or \Def{linear} if the natural projection
\[
F(X \op Y) \to F(X) \op F(Y)
\]
is an isomorphism.

We say that $F$ is \Def{quadratic} if the \Def{second cross effect}
\[
F(X|Y) \dfn \ker \left( F(X \op Y) \to F(X) \op F(Y) \right)
\]
viewed as a bifunctor is linear in both $X$ and $Y$. In this case, one has a natural decomposition
\[
F(X \op Y) \cong F(X) \op F(Y) \op F(X|Y).
\] 
\end{definition}

Proposition \ref{GammaAdditive} said that a $2$-stage $\Pi$-algebra is described by indecomposable homotopy operations, for $k \neq n-1$. There is an analogous notion in the metastable case $k=n-1$.

\begin{definition}
For $n \geq 2$, the \Def{quadratic module of indecomposables} of $\pi_{2n-1} \{ S^n \}$ is the quotient quadratic module
\[
Q_{n-1} \{ S^n \} \dfn \left( Q_{n-1,n} \ral{H} \pi_{2n-1} S^{2n-1} \ral{P} Q_{n-1,n} \right)
\]
using the notation of \ref{Indecomp}. This is well defined since $H \colon \pi_{2n-1} S^n \to \pi_{2n-1} S^{2n-1} \cong \Z$ vanishes on decomposable elements, namely compositions, since these are torsion elements.
\end{definition}

\begin{proposition} \label{GammaQuadratic}
In the metastable case $k=n-1$, the functor $\tild{\Ga}_n^{n-1}$ is the quadratic functor given by
\[
\tild{\Ga}_n^{n-1}(A_n) = A_n \ot^q Q_{n-1} \{ S^n \}.
\]
\end{proposition}

\begin{proof}
See Section \ref{sec:Proofs}.
\end{proof}

\begin{example}
In the case $n=2$ and $k=1$, we have
\[
\pi_3 \{ S^2 \} \stackrel{=}{\surj} Q_1 \{ S^2 \} \cong \left( \Z \ral{1} \Z \ral{2} \Z \right) = \Z^{\Ga}.
\]
As noted in Example \ref{WhitQuad}, the quadratic tensor product with this quadratic module is
\[
A_2 \ot^q \Z^{\Ga} \cong \Ga(A_2)
\]
which recovers the case $n=2$ of Example \ref{k1}.
\end{example}

\begin{example}
In the case $n=3$ and $k=2$, we have
\[
\pi_5 \{ S^3 \} \cong \left( \Z/2 \ral{0} \Z \ral{0} \Z/2 \right).
\]
where the group $\pi_5 S^3 \cong \Z/2$ is generated by the composite $S^5 \ral{\eta} S^4 \ral{\eta} S^3$. Therefore the quadratic module of indecomposables is
\[
Q_2 \{ S^3 \} \cong \left( 0 \to \Z \to 0 \right) = \Z^{\La}
\]
using the notation of \cite[Lemma 2.11]{Baues94}. By \cite[Proposition 4.5]{Baues94}, the quadratic tensor product with this quadratic module is the exterior square functor
\[
A_3 \ot^q \Z^{\La} \cong \La^2(A_3)
\]
which recovers the case $n=3$ of Example \ref{k2}.
\end{example}

\section{Criterion for realizability} \label{sec:Criterion}

First recall some notions and notation from \cite[\S 1,2]{Baues00}. Let $X$ be an $(n-1)$-connected CW-complex, whose homotopy $\Pi$-algebra is given inductively by the abelian group $\pi_n \dfn \pi_n X$ and maps of abelian groups
\begin{align*}
&\eta_1 \colon \Ga_n^1(\pi_n) \to \pi_{n+1} \\
&\eta_2 \colon \Ga_n^2(\eta_1) \to \pi_{n+2} \\
&\ldots \\
&\eta_k \colon \Ga_n^k(\eta_1, \eta_2, \ldots, \eta_{k-1}) \to \pi_{n+k} \\
&\ldots
\end{align*}
Note that $\eta_k$ encodes the $(n+k)$-type of $\pi_* X$.

Consider Whitehead's ``certain exact sequence'' \cite{Whitehead50}
\begin{equation} \label{CertainExSeq}
\ldots \to H_{j+1} X \ral{b} \Ga_j X \ral{i} \pi_j X \ral{h} H_j X \ral{b} \Ga_{j-1} X \to \ldots
\end{equation}
where $h$ is the Hurewicz map. There is a natural transformation $\ga$ making the diagram
\begin{equation} \label{TransformationGamma}
\xymatrix{
\Ga_n^k(\eta_1, \eta_2, \ldots, \eta_{k-1}) \ar[d]_{\ga_X} \ar[dr]^{\eta_k} & \\
\Ga_{n+k} X \ar[r]^i & \pi_{n+k} X \\
}
\end{equation}
commute. In \cite[Theorem 2.4]{Baues00}, $\ga$ is exhibited as the left edge morphism of a spectral sequence
\[
E^2_{p,q} = (L_p \Ga_n^q)(\eta_1, \eta_2, \ldots, \eta_{q-1}) \Ra \Ga_{n+p+q} X.
\]

\begin{lemma} \label{IsoGamma}
Postnikov truncation $X \to P_n X$ induces isomorphisms $\Ga_j X \ral{\cong} \Ga_j P_n X$ for $j \leq n+1$.
\end{lemma}

\begin{proof}
The truncation map $X \to P_n X$ can be chosen as a direct limit of maps $X = X_0 \to X_1 \to X_2 \to \ldots$ which are cell attachments, where $X_j \to X_{j+1}$ is attaching cells of dimension at least $n+j+2$ (in order to kill $\pi_{n+j+1}$). In particular, only cells of dimension at least $n+2$ are involved, so that with this particular cell structure, the skeleta $X^{(n+1)} = (P_n X)^{(n+1)}$ agree.

Since $\Ga_j X$ can be defined as $\Ga_j X = \im \left( \pi_j X^{(j-1)} \to \pi_j X^{(j)} \right)$ induced by skeletal inclusion, the result follows.
\end{proof}

\begin{theorem}[Criterion for realizability] \label{Criterion}
The $2$-stage $\Pi$-algebra $\ul{A}$ corresponding to
\[
\eta_k \colon \tild{\Ga}_n^k (A_n) \to A_{n+k}
\]
is realizable if and only if the map $\eta_k$ factors through the map $\ga_{K(A_n,n)}$ as illustrated in the diagram
\[
\xymatrix{
& \Ga_{n+k} K(A_n,n) \ar@{-->}[d] \\
\tild{\Ga}_n^k (A_n) \ar[ur]^{\ga_{K(A_n,n)}} \ar[r]_{\eta_k} & A_{n+k}. \\
}
\]
Here we have the isomorphism $\Ga_{n+k} K(A_n,n) \cong H_{n+k+1} K(A_n,n)$ by the Whitehead exact sequence \eqref{CertainExSeq}. The homology of Eilenberg-MacLane spaces is well known \cite{Eilenberg53} \cite{Eilenberg54} \cite{Eilenberg54III} \cite{Cartan56}.
\end{theorem}

\begin{proof}
($\Ra$) If $\ul{A}$ is realizable by a space $X$, then the natural transformation $\ga$ for $X$ yields a commutative diagram
\[
\xymatrix{
**[l] \Ga_n^k(A_n, 0, \ldots, 0) = \tild{\Ga}_n^k(A_n) \ar[d]_{\ga_X} \ar[dr]^{\eta_k} & \\
\Ga_{n+k} X \ar[r]^i & **[r]\pi_{n+k} X = A_{n+k} \\
}
\]
as noted in (\ref{TransformationGamma}). Because $X$ has $(n+k-1)$-type $P_{n+k-1} X \cong K(A_n,n)$, Lemma \ref{IsoGamma} provides a natural isomorphism
\[
\Ga_{n+k} X \cong \Ga_{n+k} (P_{n+k-1} X) \cong \Ga_{n+k} K(A_n,n)
\]
and therefore the desired factorization.

($\Leftarrow$) We will use the theorem on the realizability of the Hurewicz morphism \cite[Theorem 3.4.7]{Baues96}, starting from the $(n+k-1)$-Postnikov section of a putative realization, which is $K(A_n,n)$. Note that for $k \geq 2$, the map
\[
i_{n+k-1} \colon \Ga_{n+k-1} K(A_n,n) \to \pi_{n+k-1} K(A_n,n) = 0
\]
in Whitehead's exact sequence is null, that is, $\ker i_{n+k-1} = \Ga_{n+k-1} K(A_n,n)$. In the case $k=1$, the argument below will work anyway, using $\ker i_{n+k-1}$ instead of $\Ga_{n+k-1} K(A_n,n)$.

We are given a factorization $\eta_k = f \circ \ga_{K(A_n,n)}$, with $f \colon \Ga_{n+k} K(A_n,n) \to A_{n+k}$. Choose an epimorphism $b_1 \colon H_1 \surj \ker f$ where $H_1$ is a free abelian group. Now take $H_0 \dfn \coker f \op \Ga_{n+k-1} K(A_n,n)$ with the map $A_{n+k} \to H_0$ surjecting onto the first summand and $b_0 \colon H_0 \surj \Ga_{n+k-1} K(A_n,n)$ the projection. These maps assemble into the exact sequence
\[
H_1 \ral{b_1} \Ga_{n+k} K(A_n,n) \ral{f} A_{n+k} \to H_0 \surj \Ga_{n+k-1} K(A_n,n) \to 0. 
\]
By \cite[Theorem 3.4.7]{Baues96}, there exists a CW-complex $X$ together with a map $p \colon X \to K(A_n,n)$ inducing isomorphisms on homotopy groups $\pi_i$ for $i \leq n+k-1$ and making the diagram
\[
\resizebox{\textwidth}{!}{
\xymatrix{
H_{n+k+1} X \ar[d]_{\simeq} \ar[r] & \Ga_{n+k} X \ar[d]_{\simeq}^{p_*} \ar[r] & \pi_{n+k} X \ar[d]_{\simeq} \ar[r] & H_{n+k} X \ar[d]_{\simeq} \ar@{->>}[r] & \Ga_{n+k-1} X \ar[d]_{\simeq}^{p_*} \ar[r] & 0 \\
H_1 \ar[r]^{b_1} & \Ga_{n+k} K(A_n,n) \ar[r]^-f & A_{n+k} \ar[r] & H_0 \ar@{->>}[r] & \Ga_{n+k-1} K(A_n,n) \ar[r] & 0 \\
}
}
\]
commute, where the top row is part of Whitehead's exact sequence for $X$. By naturality of $\ga$, the diagram
\[
\xymatrix{
\tild{\Ga}_n^k(A_n) \ar@{=}[d] \ar[r]^-{\ga_X} \ar@/^1.5pc/[rr]^{\eta_k^X} & \Ga_{n+k} X \ar[d]_{\cong}^{p_*} \ar[r]^-{i_{n+k}} & \pi_{n+k} X \ar[d]_{\cong} \\
\tild{\Ga}_n^k(A_n) \ar[r]^-{\ga_{K(A_n,n)}} \ar@/_1.5pc/[rr]_{\eta_k} & \Ga_{n+k} K(A_n,n) \ar[r]^-{i_{n+k}} & A_{n+k} \\
}
\]
commutes, so that $X$ has the prescribed $\Pi$-algebra structure up to degree $n+k$. Hence the Postnikov section $P_{n+k} X$ is a realization of $\ul{A}$.
\end{proof}

\begin{corollary} \label{SplitInj}
Fix $n \geq 2$ and $k \geq 1$. Then an abelian group $A_n$ has the property that ``every $\Pi$-algebra concentrated in degrees $n$ and $n+k$ with prescribed group $A_n$ is realizable'' if and only if the map
\[
\ga_{K(A_n,n)} \colon \tild{\Ga}_n^k (A_n) \to \Ga_{n+k} K(A_n,n)
\]
is split injective.
\end{corollary}

\begin{proof}
($\Ra$) If $\ga_{K(A_n,n)}$ is \emph{not} split injective, then pick $A_{n+k} \dfn \tild{\Ga}_n^k (A_n)$ with the structure map
\[
\eta_k \dfn \id \colon \tild{\Ga}_n^k (A_n) \to \tild{\Ga}_n^k (A_n)
\]
which does not factor through $\ga_{K(A_n,n)}$, and thus defines a non-realizable $\Pi$-algebra.

($\Leftarrow$) If $\ga_{K(A_n,n)}$ is split injective, then a factorization
\[
\xymatrix{
& \Ga_{n+k} K(A_n,n) \simeq \tild{\Ga}_n^k (A_n) \op C \ar@{-->}[d]^f \\
\tild{\Ga}_n^k (A_n) \ar@{^{(}->}[ur]^{\ga_{K(A_n,n)}} \ar[r]_{\eta_k} & A_{n+k} \\
}
\]
can always be found, taking $f$ to be $\eta_k$ on the summand $\tild{\Ga}_n^k (A_n)$ and an arbitrary map on the complementary summand $C$.
\end{proof}

\begin{remark} \label{TopDetect}
As a particular case of Corollary \ref{SplitInj}, whenever $\ga$ is not injective, one can find a corresponding non-realizable $2$-stage $\Pi$-algebra. Here is another way of thinking about this.

Say that a homotopy operation $\al \in \pi_{n+k} S^n$ can be detected by a space $X$ if there is an $x \in \pi_n X$ satisfying $\al^* x \neq 0 \in \pi_{n+k} X$. Using \ref{GammaAdditive}, Theorem \ref{Criterion} says that a homotopy operation $\al \in Q_{k,n}$ can be detected by a $2$-stage space if and only if it satisfies $\ga_{K(\Z,n)}(\al) \neq 0$. Indeed, one has the realizable $2$-stage $\Pi$-algebra $\ul{A}$ with $A_n = \Z$, $A_{n+k} = \Ga_{n+k} K(\Z,n)$, and $\ga_{K(\Z,n)} \colon Q_{k,n} \to \Ga_{n+k} K(\Z,n)$ as structure map.
\end{remark}

\begin{remark}
In principle, the obstruction to realizability exhibited in \ref{Criterion} could be interpreted in terms of an obstruction class in Andr\'e-Quillen cohomology of the $\Pi$-algebra $\ul{A}$ \cite{Blanc04} \cite{Frankland11}, or equivalently, in terms of higher homotopy operations \cite{Blanc12}.
\end{remark}

\subsection*{Relationship to $k$-invariants}

It is a classic fact that connected spaces are classified up to homotopy by their $k$-invariants. In particular, a $2$-stage space $X$ with homotopy groups $\pi_n \dfn \pi_n X$ and $\pi_{n+k} \dfn \pi_{n+k} X$ (where $n \geq 2$) is classified by its $k$-invariant
\[
\ka \in H^{n+k+1} \left( K(\pi_n, n); \pi_{n+k} \right).
\]
Via the natural surjective map
\[
\te \colon H^{n+k+1} \left( K(\pi_n, n); \pi_{n+k} \right) \surj \Hom_{\Z} \left( H_{n+k+1}(K(\pi_n, n),\Z), \pi_{n+k} \right)
\]
this yields a map of abelian groups
\[
\Ga_{n+k} K(\pi_n, n) \cong H_{n+k+1}(K(\pi_n, n),\Z) \ral{\te(\ka)} \pi_{n+k}.
\]
Another point of view on Theorem \ref{Criterion}, as well as an alternate proof, is that the $\Pi$-algebra $\pi_* X$ is given by the structure map
\[
\xymatrix{
\tild{\Ga}_n^k(\pi_n) \ar[r]^-{\ga_{K(\pi_n,n)}} \ar@/_1.5pc/[rr]_{\eta_k} & \Ga_{n+k} K(\pi_n,n) \ar[r]^-{\te(\ka)} & \pi_{n+k}. \\
}
\]
This follows from the theorem on $k$-invariants in \cite[Theorem 2.5.10 (b)]{Baues96} and diagram \eqref{TransformationGamma}. Therefore, the realizable $2$-stage $\Pi$-algebras are precisely those whose structure map $\eta_k$ factors through $\ga_{K(\pi_n,n)}$.

\section{Stable case} \label{sec:Stable}

A $\Pi$-algebra concentrated in a stable range $n, n+1, \ldots, n+k$ with $k \leq n-2$ can be identified with a module over the stable homotopy ring $\pi_*^S$, or more precisely its Postnikov truncation $\pi_{* \leq k}^S$. Indeed, in such a $\Pi$-algebra $\ul{A}$, all Whitehead products vanish for dimension reasons, and all precomposition operations $\al^* \colon A_{n+i} \to A_{n+j}$ are induced by maps $\al \colon S^{n+j} \to S^{n+i}$ that live in stable homotopy groups $\pi_{j-i}^S$. The identification is made more precise in \ref{StableRange}.

\begin{proposition}
A $\Pi$-algebra concentrated in a stable range $n, n+1, \ldots, n+k$ is realizable (by a space) if and only if the corresponding $\pi_*^S$-module is realizable (by a spectrum).
\end{proposition}

\begin{proof}
($\Ra$) Let $\ul{A}$ be a $\Pi$-algebra concentrated in said stable range, and denote also by $A$ the corresponding $\pi_*^S$-module. If $X$ is a space realizing $\ul{A}$, then the Postnikov truncation $P_{n+k} \Si^{\infty} X$ of the suspension spectrum of $X$ is a spectrum realizing $A$.

($\Leftarrow$) Let $M$ be a $\pi_*^S$-module concentrated in a stable range, so that the corresponding $\Pi$-algebra is $\Om^{\infty} M$, by \ref{StableRange}. If $Z$ is a spectrum realizing $M$, then the infinite loop space $\Om^{\infty} Z$ is a space realizing $\Om^{\infty} M$, by \ref{LoopsInfty}.
\end{proof}

\begin{remark}
A $\pi_*^S$-module $M$ is realizable if and only if any of its shifts $\Si^j M$ (for $j \in \Z$) is realizable. This follows from the isomorphism $\pi_* (\Si^j Z) \cong \Si^j (\pi_* Z)$ of $\pi_*^S$-modules.
\end{remark}

The criterion \ref{Criterion} indicates that the map
\[
\ga_{K(A_n,n)} \colon \tild{\Ga}_n^k (A_n) \to \Ga_{n+k} K(A_n,n) \cong H_{n+k+1} K(A_n,n)
\]
plays a key role for determining realizability. In the stable range $k \leq n-2$, we have seen in \ref{GammaAdditive} that the domain of $\ga_{K(A_n,n)}$ is
\[
\tild{\Ga}_n^k (A_n) = A_n \ot Q_k^S
\]
while its codomain is
\[
H_{n+k+1} K(A_n,n) \cong (H \Z)_{k+1} (H A_n) \cong (H A_n)_{k+1} (H \Z) 
\]
where $HA$ denotes the Eilenberg-MacLane spectrum of an abelian group $A$. The universal coefficient theorem yields a natural exact sequence
\[
0 \to A_n \ot H \Z_{k+1} H \Z \inj (H A_n)_{k+1} H \Z \surj \Tor_1^{\Z} (A_n, H \Z_k H \Z) \to 0
\]
which is split (non-naturally).

\begin{lemma} \label{KanExtn}
Let $R$ be a commutative ring, $R\Mod$ the category of $R$-modules, and $\io \colon R\Modff \to R\Mod$ the inclusion of the full subcategory of finitely generated free $R$-modules.

Let $F \colon R\Modff \to R\Mod$ be an additive functor. Then there is a unique (up to unique natural isomorphism) extension $\ol{F} \colon R\Mod \to R\Mod$ of $F$ which preserves all (small) colimits. Moreover, $\ol{F}$ is natural in $F$. It is given by $\ol{F} = - \ot_R FR$. For any functor $G \colon R\Mod \to R\Mod$, there is a natural transformation $\ol{\io^* G} \to G$, which is natural in $G$.
\end{lemma}

\begin{proof}
The left Kan extension $\ol{F} = \Lan_{\io} F$ satisfies all the properties in the statement.
\end{proof}

\begin{remark}
The functor $\ol{\io^*G}$ is \emph{not} the $0^{\text{th}}$ left derived functor $L_0 G$ of $G$, which provides the best approximation of $G$ by a right exact functor, with comparison map $L_0 G \to G$. Indeed, there exist additive right exact functors $\Ab \to \Ab$ which do not preserve infinite direct sums. However, the comparison maps do fit together as $\ol{\io^* G} \to L_0 G \to G$.
\end{remark}

\begin{proposition} \label{TensorSummand}
In the stable range $k \leq n-2$, the map
\[
\ga_{K(A_n,n)} \colon A_n \ot Q_k^S \to (H \Z)_{k+1} (H A_n)
\]
factors through the summand $A_n \ot H \Z_{k+1} H \Z$, that is, we have
\[
\ga_{K(A_n,n)} \colon A_n \ot Q_k^S \to A_n \ot H \Z_{k+1} H \Z \inj (H \Z)_{k+1} (H A_n).
\]
\end{proposition}

\begin{proof}
First, note that the assignment $A \mapsto H \Z_{k+1} HA$ defines an additive functor $G \colon \Ab \to \Ab$. Indeed, for abelian groups $A,B$, we have:
\begin{align*}
G(A \op B) &= H \Z_{k+1} H(A \op B) \\
&\cong H \Z_{k+1} (HA \vee HB) \\
&\cong H \Z_{k+1} HA \op H \Z_{k+1} HB \\
&= G(A) \op G(B).
\end{align*}
Now $\ga \colon F \to G$ is a natural transformation from the functor $F = - \ot Q_k^S$ to $G$ and, by Lemma \ref{KanExtn}, induces a commutative diagram
\[
\xymatrix{
\ol{\io^* F} \ar[d]_{\ep_F} \ar[r]^{\ol{\io^* \ga}} & \ol{\io^* G} \ar[d]^{\ep_G} \\
F \ar[r]_{\ga} & G. \\
}
\]
Because $F$ is of the form $F = - \ot F \Z$, it preserves all colimits, and thus $\ep_F$ is an isomorphism. Moreover we have
\[
\ol{\io^* G} = - \ot G \Z = - \ot H\Z_{k+1} H\Z
\]
and the coaugmentation
\[
(\ep_G)_A \colon A \ot H\Z_{k+1} H\Z \to HA_{k+1} H\Z
\]
is the usual inclusion of the tensor summand. Therefore $\ga$ factors through said inclusion.
\end{proof}

\begin{corollary}
In the stable range $k \leq n-2$, every $\Pi$-algebra concentrated in degrees $n$ and $n+k$ is realizable if and only if the map 
\[
\ga_{K(\Z,n)} \colon Q_k^S \to H \Z_{k+1} H \Z
\]
is split injective. Note that the map does not depend on $n$, only on the stable stem $k$.
\end{corollary}

\begin{proof}
By \ref{SplitInj}, every $\Pi$-algebra concentrated in degrees $n$ and $n+k$ is realizable if and only if the maps 
\[
\ga_{K(A_n,n)} \colon A_n \ot Q_k^S \to (H \Z)_{k+1} (H A_n)
\]
are split injective for every abelian group $A_n$. By \ref{TensorSummand}, this is equivalent to the maps
\[
\ga_{K(A_n,n)} \colon A_n \ot Q_k^S \to A_n \ot H \Z_{k+1} H \Z
\]
being split injective. Since applying $A_n \ot -$ (or any functor) to a split monomorphism yields a split monomorphism, this is equivalent to the single map 
\[
\ga_{K(\Z,n)} \colon Q_k^S \to H \Z_{k+1} H \Z
\]
being split injective.
\end{proof}

\section{Non-realizable examples} \label{sec:NonRealiz}

As noted in Example \ref{k1} and Proposition \ref{k2Realizable}, all $2$-stage $\Pi$-algebras with stem $k=1$ or $k=2$ are realizable -- for any value of $n$, not only stably. We will show that the smallest stem where a non-realizable example appears is $k=3$.

Let us recall the first few stable homotopy groups of spheres; see \cite[\S 4]{Baues00}. In degrees $* \leq 6$, the stable homotopy ring $\pi_*^S$ is generated (as an algebra) by elements $\eta \in \pi_1^S$, $\nu \in \pi_3^S$, and $\al \in \pi_3^S$, subject to relations
\begin{align*}
&2 \eta = 0\\
&4 \nu = \eta^3 \\
&\eta \nu = 0 \\
&2 \nu^2 = 0 \\
&3 \al = 0 \\
&\al^2 = 0.
\end{align*}
Here $\eta$ is the stabilization of the Hopf map $S^3 \to S^2$ and $\nu$ is the $2$-primary part of the stabilization of the Hopf map $H \colon S^7 \to S^4$. Integrally, $\nu$ can be thought of as, say, $3H$. The element $\al$ is the first in the $3$-primary alpha family.

The first few stable homotopy groups are
\[
\pi_i^S = \begin{cases}
\Z &i = 0 \\
\Z/2 \lan \eta \ran &i = 1 \\
\Z/2 \lan \eta^2 \ran &i = 2 \\
\Z/24 \simeq \Z/8 \lan \nu \ran \op \Z/3 \lan \al \ran &i = 3 \\
0 &i = 4 \\
0 &i = 5 \\
\Z/2 \lan \nu^2 \ran &i = 6 \\
\end{cases}
\]
and their indecomposables are
\[
Q_i^S = \begin{cases}
\Z &i = 0 \\
\Z/2 \lan \eta \ran &i = 1 \\
0 &i = 2 \\
\Z/12 \simeq \Z/4 \lan \nu \ran \op \Z/3 \lan \al \ran &i = 3 \\
0 &i = 4 \\
0 &i = 5 \\
0 &i = 6. \\
\end{cases}
\]

\begin{proposition} \label{Smallest}
Let $n \geq 5$. The (stable) $\Pi$-algebra $\ul{A}$ concentrated in degrees $n$ and $n+3$ given by $A_n = \Z$ and $A_{n+3} = \Z/4$ with structure map $\eta_3 \colon A_n \ot Q_3^S \to A_{n+3} = \Z/4$ given by the projection
\[
A_n \ot Q_3^S \cong Q_3^S = \Z/4 \lan \nu \ran \op \Z/3 \lan \al \ran \surj \Z/4 
\]
sending $\nu$ to $1$ is not realizable.
\end{proposition}

\begin{proof}
According to \cite[Theorem 25.1]{Eilenberg54}, we have $H\Z_4 H\Z \simeq \Z/6 = \Z/2 \op \Z/3$. Therefore the map $\ga \colon Q_3^S \simeq \Z/12 \to \Z/6 \simeq H\Z_4 H\Z$ sends $2 \nu$ to $0$, whereas $\eta_3$ does not. The result follows from \ref{Criterion}.
\end{proof}

Theorem \ref{Criterion} reduces realizability questions to the algebraic problem of understanding the map $\ga$, but it can also be used the other way around. In the following proposition, we start from a realizable $2$-stage $\Pi$-algebra and deduce information about the map $\ga$ using Theorem \ref{Criterion}.

\begin{proposition}
The map $\ga \colon Q_3^S \to H\Z_4 H\Z$ sends $\al$ to a non-zero element (therefore of order $3$).
\end{proposition}

\begin{proof}
Take $n \geq 5$ and consider the localization at $3$ of the sphere $S^n \to S^n_{(3)}$, then take Postnikov sections $P_{n+3} S^n \to P_{n+3} S^n_{(3)} =: X$. Because this map induces $3$-localization on homotopy groups (and a map of $\Pi$-algebras), the $\Pi$-algebra $\pi_* X$ consists of two non-zero groups
\begin{align*}
&\pi_n X \cong \Z_{(3)} \\
&\pi_{n+3} X \cong \Z/3 \lan \al \ran
\end{align*}
with structure map
\[
\eta_3 \colon \pi_n X \ot Q_3^S \ral{\simeq} \pi_{n+3} X  
\]
sending $\al$ to $\al$, i.e. the identity via the identification
\[
\pi_n X \ot Q_3^S \cong \Z_{(3)} \ot \left( \Z/4 \lan \nu \ran \op \Z/3 \lan \al \ran \right) = \Z/3 \lan \al \ran.
\]
By \ref{Criterion}, we deduce that the map
\[
\Z_{(3)} \ot \ga \colon \Z_{(3)} \ot Q_3^S \cong \Z/3 \lan \al \ran \to \Z_{(3)} \ot H \Z_4 H \Z \simeq \Z/3
\]
sends $\al$ to a non-zero element, and therefore so does $\ga$.
\end{proof}

In fact, the same argument yields a more general statement.

\begin{proposition} \label{Alpha1}
Fix a prime $p \geq 3$ and consider the Greek letter element $\al_1 \in Q_{2(p-1)-1}^S$. The map $\ga \colon Q_{2(p-1)-1}^S \to H\Z_{2(p-1)} H\Z$ sends $\al_1$ to a non-zero element (therefore of order $p$).
\end{proposition}

\begin{proof}
Write the stable stem $k \dfn \abs{\al_1} = 2(p-1)-1$ and take $n$ very large, namely $n \geq k+2$. Consider the localization at $p$ of the sphere $S^n \to S^n_{(p)}$, then take Postnikov sections $P_{n+k} S^n \to P_{n+k} S^n_{(p)} =: X$.

A key feature of $\al_1$ is that it generates $\pi_{2p-3}^S \ot \Z_{(p)} \simeq \Z/p$ and is the first element of order a power of $p$ in $\pi_*^S$ \cite[(13.4)]{Toda62}. Thus the $p$-localization of all lower (positive) stems is zero. Therefore the $\Pi$-algebra $\pi_* X$ consists of two non-zero groups
\begin{align*}
&\pi_n X \cong \Z_{(p)} \\
&\pi_{n+k} X \cong (\pi_k^S)_{(p)} \simeq \Z/p
\end{align*}
in which $\al_1$ is detected. More precisely, taking $1 \in \pi_n X$ we have $\al_1^*(1) = \al_1 \neq 0$ in $\pi_{n+k} X$. By \ref{Criterion} (and Remark \ref{TopDetect}), $\ga$ sends $\al_1$ to a non-zero element.
\end{proof}

\subsection*{Infinite families}

Proposition \ref{Smallest} provides a non-realizable $2$-stage $\Pi$-algebra with the lowest possible stem dimension $k=3$. Our next goal is to find an infinite family of such examples, in infinitely many stem dimensions $k$. For this we need an infinite family of indecomposables in $Q_*$. The Greek letter elements, for example the $\al$ and $\be$ families, are good candidates.

The next proposition provides non-realizable examples using a different method: finding elements of homotopy groups of spheres which are indecomposable as primary operations, but decomposable as secondary operations.

\begin{proposition} \label{Alpha}
Fix a prime $p \geq 3$ and consider the alpha elements $\al_i \in Q_{2i(p-1)-1}^S$ \cite[Definition 1.3.10, Theorem 1.3.11]{Ravenel86}. For every $i \geq 2$, the map $\ga \colon Q_{2i(p-1)-1}^S \to H\Z_{2i(p-1)} H\Z$ sends $\al_i$ to zero.
\end{proposition}

\begin{proof}
For $i \geq 2$, there is a Toda bracket \cite[(13.4)]{Toda62}
\[ 
\al_i \in \lan \al_1, p, \al_{i-1} \ran
\]
so that $\al_i$ cannot be detected by a $2$-stage space (or spectrum), and by \ref{TopDetect} we have $\ga(\al_i) = 0$.

In more detail, write $s = \abs{\al_1}$ and $t = \abs{\al_{i-1}}$ so that $\abs{\al_i} = s+t+1$, and assume $X$ is a space with homotopy concentrated in degrees $n$ and $n+s+t+1$ (for $n$ large). Let us illustrate the Toda bracket setup:
\[
S^{n+s+t} \ral{\al_{i-1}} S^{n+s} \ral{p} S^{n+s} \ral{\al_1} S^n.
\]
Pick any $x \in \pi_n X$. We claim that the precomposition $\al_i^*(x) = x \al_i$ is null. Postcomposing by $x$ defines a map \cite[Proposition 1.2 (iv)]{Toda62}
\begin{align*}
\lan \al_1, p, \al_{i-1} \ran \ral{x \circ -} &\lan x \al_1, p, \al_{i-1} \ran \\
= &\lan 0, p, \al_{i-1} \ran
\end{align*}
using the fact $x \al_1 \in \pi_{n+s} X = 0$. The indeterminacy of $\lan 0, p, \al_{i-1} \ran$ is
\begin{align*}
0 [S^{n+s+t+1}, S^{n+s}] &+ [S^{n+s+1}, X] \al_{i-1} \\
&= (\pi_{n+s+1} X) \al_{i-1} \\
&= \{ 0 \}
\end{align*}
again using the assumption on $\pi_* X$. Moreover, $0$ is clearly a representative in $\lan 0, p, \al_{i-1} \ran$ \cite[Proposition 1.2 (0)]{Toda62}, thus we have equality $\lan 0, p, \al_{i-1} \ran = \{ 0 \}$. Therefore $x \al_i \in \lan 0, p, \al_{i-1} \ran$ is null, as claimed.
\end{proof}


\begin{proposition} \label{DividedAlpha}
Fix a prime $p \geq 3$ and consider the divided alpha elements $\al_{i/j} \in Q_{2i(p-1)-1}^S$, where $j \leq \nu_p(i) + 1$, and $\nu_p$ denotes the $p$-adic valuation \cite[Definition 1.3.19]{Ravenel86}. For every $j \geq 2$, we have $p \al_{i/j} \neq 0$ but $\ga (p \al_{i/j}) = 0$.
\end{proposition}

\begin{proof}
Recall a few properties of the divided alpha elements \cite{Ravenel86} \cite[\S 1]{Behrens09}. The element
\[
\al_{i/j} \in \Ext_{BP_* BP}^{1, 2i(p-1)}(BP_*, BP_*)
\]
defined in the $E_2$-term of the Adams-Novikov spectral sequence is a permanent cycle and therefore represents an element in homotopy $\al_{i/j} \in \pi_{2i(p-1)-1}^S$ which is known to be in the image of the $J$-homomorphism. It has (additive) order $p^j$, is indecomposable, and its order in $Q_*^S$ is still $p^j$. This proves $p \al_{i/j} \neq 0$ in $Q_*^S$.

On the other hand, the $p$-torsion in $H\Z_*H\Z$ is annihilated by a single power of $p$ \cite[Theorem 3.1]{Kochman82} \cite[\S 11, Theorem 2]{Cartan56}. Therefore the map $\ga \colon Q_*^S \to H\Z_{*+1}H\Z$ must send $p \al_{i/j}$ to zero.
\end{proof}

\begin{remark}
In Proposition \ref{DividedAlpha}, we may as well take $i = p^{j-1}$.
\end{remark}

Whenever $\ga \colon Q_k^S \to H\Z_{k+1} H\Z$ is non-injective, we can find a corresponding non-realizable $2$-stage $\Pi$-algebra in stem dimension $k$. Therefore, Propositions \ref{Alpha} and \ref{DividedAlpha} provide infinite families of non-realizable examples, in infinitely many stem dimensions.

Note that \cite[Theorem 8.1]{Blanc95} also provides a (different) infinite family of non-realizable $\Pi$-algebras, which can be truncated to two non-zero degrees. The argument used there is similar to that of \ref{Alpha}.

\subsection*{A $3$-stage example}

\begin{proposition}
The stable $3$-stage $\Pi$-algebra $\ul{A}$ defined by $A_n = A_{n+1} = A_{n+2} = \Z/2$ (where $n \geq 4$) with structure maps
\begin{align*}
&\eta_1 \colon \Ga_n^1 (A_n) = A_n \ot \Z/2 = \Z/2 \ral{\cong} \Z/2 = A_{n+1} \\
&\eta_2 \colon \Ga_n^2 (A_n, \eta_1) = A_{n+1} \ot \Z/2 = \Z/2 \ral{\cong} \Z/2 = A_{n+2}
\end{align*}
is non-realizable.
\end{proposition}

\begin{proof}
The map $E_n(\eta_1)$ described in \cite[\S 3.2]{Baues00} is the composite
\[
\resizebox{\textwidth}{!}{
\xymatrix{
\Tor(A_n,\Z/2) \ar@{^{(}->}[r]^-{i} & A_n \ar@{->>}[r]^-{q} & A_n \ot \Z/2 \ar[r]^-{\eta_1} & A_{n+1} \ar@{->>}[r]^-{q} & A_{n+1} \ot \Z/2 \cong \Ga_n^2 (A_n,\eta_1) \\
}
}
\]
which in our case is the isomorphism
\[
\xymatrix{
\Z/2 \ar@{^{(}->}[r]^-{i}_-{\cong} & \Z/2 \ar@{->>}[r]^-{q}_-{\cong} & \Z/2 \ar[r]^-{\eta_1}_-{\cong} & \Z/2 \ar@{->>}[r]^-{q}_-{\cong} & \Z/2. \\
}
\]
The obstruction $O(A) = \eta_2 \circ E_n(\eta_1)$ described in \cite[Theorem 3.3 (B)]{Baues00} is the non-zero map $\Z/2 \ral{\cong} \Z/2 \ral{\cong} \Z/2$. Therefore $\ul{A}$ is non-realizable. 
\end{proof}

\begin{remark}
By contrast, the example in \cite[Example 7.18]{Blanc95} with the same homotopy groups but a different $\Pi$-algebra structure is in fact realizable.
\end{remark}

\section{Proofs} \label{sec:Proofs}

\subsection*{Theories and $\pi_*^S$-modules}

The category $\PI$ forms a \emph{theory} in the sense of Lawvere \cite[\S 6]{Baues99}, more precisely a \emph{graded} (or \emph{multisorted}) \emph{theory} \cite[\S 8]{Baues99}. We adopt the following convention.

\begin{definition}
A \Def{theory} is a category with finite coproducts, including the empty coproduct (initial object $*$).

Let $\TT$ be a theory. A \Def{model} for $\TT$ is a product-preserving functor $\TT^{\opp} \to \Set$, in other words, a contravariant functor sending coproducts to products.

As in \cite[\S 1]{Baues00}, let $\model(\TT):= \Fun^{\x}(\TT^{\opp}, \Set)$ denote the category of models for a theory $\TT$.
\end{definition}

In this terminology, $\Pi$-algebras are models for $\PI$, or in symbols: $\PiAlg = \model(\PI)$. Note that $\PI_n$ and $\PI_n^k$ are also theories, and the inclusion functors $\PI_n^k \to \PI_n \to \PI$ are maps of theories, i.e., preserve coproducts. The equivalences $\PiAlg_n \cong \model(\PI_n)$ and $\PiAlg_n^k \cong \model(\PI_n^k)$ are proved in \cite[Proposition 4.5, Remark 4.6]{Frankland11}.

Let us study the stable case as in Section \ref{sec:Stable} more precisely. Given a spectrum $Z$, its homotopy groups $\pi_* Z$ naturally form a $\pi_*^S$-module, where $\pi_*^S$ is the stable homotopy ring. This algebraic structure can also be described as a model for a theory.

\begin{notation}
Let $\Ho \Sp$ denote the stable homotopy category \cite[\S 2.2]{Margolis83} and let $\PIst$ denote its full subcategory consisting of finite wedges of sphere spectra $\vee S^{n_i}$, $n_i \in \Z$. Here again, the empty wedge (a point) is allowed.
\end{notation}

We have the isomorphism of categories $\model(\PIst) \cong \pi_*^S\Mod$, sending a model $M$ to the $\pi_*^S$-module with $i^{\text{th}}$ graded piece $M_i \dfn M(S^i)$, endowed with the induced precomposition operations. Given a spectrum $Z$, the realizable $\pi_*^S$-module $\pi_* Z$ corresponds to the functor $[-,Z]$.

We can now make the relationship between $\Pi$-algebras and $\pi_*^S$-modules precise. Consider the suspension spectrum functor $\Si^{\infty} \colon \PI \to \PIst$ which sends maps to their stabilization. Because $\Si^{\infty}$ preserves coproducts (wedges), it induces a restriction functor on models
\[
\Om^{\infty} \dfn (\Si^{\infty})^* \colon \pi_*^S\Mod \to \PiAlg.
\]
Concretely, $\Om^{\infty} M$ has the same underlying graded group as $M$ in degrees $i \geq 1$, and maps between spheres act on $\Om^{\infty} M$ via their stabilization. The notation $\Om^{\infty}$ is justified by the following proposition.

\begin{proposition} \label{LoopsInfty}
For any spectrum $Z$, there is an isomorphism of $\Pi$-algebras $\pi_* (\Om^{\infty} Z) \cong \Om^{\infty} (\pi_* Z)$, which is natural in $Z$.
\end{proposition}

\begin{proof}
Let $S$ be an object of $\PI$, that is, a finite wedge of spheres. By definition, we have:
\begin{align*}
&\pi_* (\Om^{\infty} Z)(S) = [S,\Om^{\infty} Z] \\
&\Om^{\infty} (\pi_* Z)(S) = (\pi_* Z) (\Si^{\infty} S) = [\Si^{\infty} S, Z].
\end{align*}
Moreover, $\Si^{\infty}$ is left adjoint to $\Om^{\infty}$ so that we have an isomorphism of sets
\[
[S,\Om^{\infty} Z] \cong [\Si^{\infty} S, Z]
\]
which is natural in $S$ and $Z$. Naturality in $S$ provides the isomorphism of $\Pi$-algebras $\pi_* (\Om^{\infty} Z) \simeq \Om^{\infty} (\pi_* Z)$, while naturality in $Z$ implies that this isomorphism of $\Pi$-algebras is also natural.
\end{proof}

Consider the full subcategories $(\PIst)_n$ and $(\PIst)_n^k$ of $\PIst$, which are themselves theories. As in the unstable picture, the inclusion functors $(\PIst)_n^k \to (\PIst)_n \to \PIst$ are maps of theories. Here again, there are isomorphisms of categories $\pi_*^S \Mod_n \cong \model((\PIst)_n)$ and $\pi_*^S \Mod_n^k \cong \model((\PIst)_n^k)$.

\begin{proposition} \label{StableRange}
In the stable range $k \leq n-2$, the functor $\Om^{\infty}$ restricts to an equivalence of categories
\[
\Om^{\infty} \colon \pi_*^S \Mod_n^k \ral{\cong} \PiAlg_n^k.
\]
\end{proposition}

\begin{proof}
In the stable range, the stabilization functor $\Si^{\infty} \colon \PI_n^k \to (\PIst)_n^k$ is an equi\-va\-lence of categories. Therefore, it induces an equivalence on models
\[
(\Si^{\infty})^* \colon \model((\PIst)_n^k) \ral{\cong} \model(\PI_n^k)
\]
which is the desired equivalence.
\end{proof}

\subsection*{Split linear extension of theories}

\begin{proposition} \label{SplitLinExtn}
Let $n \geq 2$ and $k \geq 1$. Consider the functor
\begin{align*}
D \colon (\PI_{n+k}^0)^{\opp} \x \PI_n^{k-1} &\to \Ab \\
(S,U) &\mapsto [S,U].
\end{align*}
Then the theory $\PI_n^k$ with its natural projection
\[
\PI_n^k \to \PI_{n+k}^0 \x \PI_n^{k-1}
\]
given by ``collapse'' functors \cite[\S 4]{Frankland11} is the split linear extension \cite[Definition 7.1]{Baues99} of $\PI_{n+k}^0 \x \PI_n^{k-1}$ by $D$. 
\end{proposition}

\begin{proof}
Note that $D$ takes values in $\Ab$ because every object $S = \vee_i S^{n+k}$ of $\PI_{n+k}^0$ is an abelian cogroup object (of $\PI$ or $\PI_n^k$). Moreover, $D$ is additive in $\PI_{n+k}^0$:
\[
D(S_1 \vee S_2, U) = [S_1 \vee S_2, U] = [S_1, U]_* \x [S_2, U] = D(S_1,U) \x D(S_2,U)
\]
and satisfies $D(S,\ast) = [S,\ast] = 0$ for any $S \in \PI_{n+k}^0$. Therefore, there is such a thing as the split linear extension $\TT$ of $\PI_{n+k}^0 \x \PI_n^{k-1}$ by $D$, with its projection $q \colon \TT \to \PI_{n+k}^0 \x \PI_n^{k-1}$. 

Let us construct an equivalence of categories $\phy \colon \PI_n^k \ral{\cong} \TT$ with inverse $\psi \colon \TT \ral{\cong} \PI_n^k$. Note that every object $X$ of $\PI_n^k$, i.e. a finite wedge of spheres of dimensions from $n$ to $n+k$, can be uniquely expressed as a wedge $X = S \vee U$ with $S \in \PI_{n+k}^0, U \in \PI_n^{k-1}$, i.e. $S$ contains the spheres of dimension $n+k$ and $U$ contains the remaining spheres, of dimensions from $n$ to $n+k-1$. Moreover, extracting either summand from $X$ is functorial in $X$, using the collapse functors
\begin{align*}
\col^{\hi} \colon \PI_n^k \to \PI_{n+k}^0 \\
\col^{\lo} \colon \PI_n^k \to \PI_n^{k-1}
\end{align*}
which extract the spheres of highest dimension $n+k$ and lower dimensions $n$ to $n+k-1$, respectively. By abuse of notation, write $\col^{\hi} \colon X \surj S$ and $\col^{\lo} \colon X \surj U$ for the corresponding collapse maps.

\textbf{Step 1: Construction of $\phy \colon \PI_n^k \to \TT$.} On objects, take
\[
\phy(X \cong S \vee U) := (S,U) = (\col^{\hi} X, \col^{\lo} X)
\]
and for a morphism $X_1 \cong S_1 \vee U_1 \ral{f} S_2 \vee U_2 \cong X_2$, $\phy(f)$ is defined by the data
\[
\begin{cases}
&S_1 \stackrel{\inc_1^{\hi}}{\inj} S_1 \vee U_1 \ral{f} S_2 \vee U_2 \stackrel{\col_2^{\hi}}{\surj} S_2 \\
&U_1 \stackrel{\inc_1^{\lo}}{\inj} S_1 \vee U_1 \ral{f} S_2 \vee U_2 \stackrel{\col_2^{\lo}}{\surj} U_2 \\
&S_1 \stackrel{\inc_1^{\hi}}{\inj} S_1 \vee U_1 \ral{f} S_2 \vee U_2 \stackrel{\col_2^{\lo}}{\surj} U_2 \\
\end{cases}
\]
where the last piece of data is an element of $[S_1,U_2]_* = D(S_1,U_2)$. In symbols:
\begin{align*}
\phy(f) &= \left( \col^{\hi}(f), \col^{\lo}(f), \col_2^{\lo} \circ f \circ \inc_1^{\hi} \right) \\
&=: \left( f^{\hi}, f^{\lo}, f^{\hi \lo} \right).
\end{align*}
We have $\phy (\id_X) = \id_{\phy X} = (\id_S, \id_U, 0)$. Remains to check that $\phy$ respects composition. Given a composite $X_1 \ral{f} X_2 \ral{g} X_3$ in $\PI_r^k$, which we write as
\[
S_1 \vee U_1 \ral{f} S_2 \vee U_2 \ral{g} S_3 \vee U_3
\]
applying $\phy$ yields
\begin{align*}
\phy(gf) &= \left( (gf)^{\hi}, (gf)^{\lo}, (gf)^{\hi \lo} \right) \\
&= \left( g^{\hi} f^{\hi}, g^{\lo} f^{\lo}, (gf)^{\hi \lo} \right)
\end{align*}
whereas the composite in $\TT$ is
\begin{align*}
\phy(g) \phy(f) &= \left( g^{\hi}, g^{\lo}, g^{\hi \lo} \right) \left( f^{\hi}, f^{\lo}, f^{\hi \lo} \right) \\
&= \left( g^{\hi} f^{\hi}, g^{\lo} f^{\lo}, (f^{\hi})^* g^{\hi \lo} + (g^{\lo})_* f^{\hi \lo} \right).
\end{align*}
A straightforward calculation proves the equality $(gf)^{\hi \lo} = (f^{\hi})^* g^{\hi \lo} + (g^{\lo})_* f^{\hi \lo}$.

\textbf{Step 2: Construction of $\psi \colon \TT \to \PI_n^k$.} On objects, take
\[
\psi(S,U) := S \vee U
\]
and for a morphism
\[
(f^h,f^l,\de) \colon (S_1,U_1) \to (S_2,U_2)
\]
in $\TT$, with $\de \in D(S_1,U_2) = [S_1,U_2]$, define the morphism
\begin{align*}
&\psi(f^h,f^l,\de) \colon S_1 \vee U_1 \to S_2 \vee U_2 \\
&\psi(f^h,f^l,\de) = \left( \inc_2^{\hi} f^h + \inc_2^{\lo} \de \right) ; \inc_2^{\lo} f^l.
\end{align*}
We have
\[
\psi 1_{(S,U)} = \psi (1_S,1_U,0) = \inc^{\hi} \vee \inc^{\lo} = 1_{S \vee U}
\]
and it remains to check that $\psi$ respects composition. Given a composite
\[
\xymatrix{
(S_1, U_1) \ar[r]^{(f^h, f^l, \de)} \ar@/_2pc/[rr]_-{(g^h f^h, g^l f^l, (f^h)^* \ep + (g^l)_
* \de)} & (S_2, U_2) \ar[r]^{(g^h, g^l, \ep)} & (S_3, U_3) \\
}
\]
in $\TT$, applying $\psi$ yields
\[
\xymatrix{
S_1 \vee U_1 \ar[rr]^-{\inc_2^{\hi} f^h + \inc_2^{\lo} \de; \inc_2^{\lo} f^l} \ar@/_2pc/[rrrr]_-{\inc_3^{\hi} g^h f^h + \inc_3^{\lo} ((f^h)^* \ep + (g^l)_* \de); \inc_3^{\lo} g^l f^l} & & S_2 \vee U_2 \ar[rr]^-{\inc_3^{\hi} g^h + \inc_3^{\lo} \ep; \inc_3^{\lo} g^l} & & S_3 \vee U_3 \\
}
\]
which is still commutative. This follows from right distributivity for maps between spheres \cite[Theorem X.8.1]{Whitehead78}, as well as  Hilton's formula \cite[Theorem XI.8.5]{Whitehead78} \cite[\S A.9]{Baues96} and the fact that $f^h \colon S_1 \to S_2$ is a map between spheres of equal dimensions (namely $n+k$). In that case, the Hilton--Hopf invariants vanish and composition is in fact left distributive, in other words precomposition by $f^h$ is linear.

\textbf{Step 3: $\psi \phy = \id_{\PI_n^k}$.} On objects, the composite of functors does
\[
(X \cong S \vee U) \stackrel{\phy}{\mapsto} (S,U) \stackrel{\psi}{\mapsto} S \vee U
\]
and on a map $X_1 \cong S_1 \vee U_1 \ral{f} S_2 \vee U_2 \cong X_2$, the composite does
\begin{align*}
f \stackrel{\phy}{\mapsto} &\left( f^{\hi}, f^{\lo}, f^{\hi \lo} \right) \\
\stackrel{\psi}{\mapsto} &\left( \inc_2^{\hi} f^{\hi} + \inc_2^{\lo} f^{\hi \lo} \right) ; \inc_2^{\lo} f^{\lo}.
\end{align*}
Here comes the topological argument. Note that $S$ is $(n+k-1)$-connected and $U$ is $(n-1)$-connected, so that the natural map $S \vee U \to S \x U$ is $(n+k+n-1)$-connected. This implies that for $i \leq n+k+n-2$ (in particular for $i \leq n+k$), any map $g \colon S^i \to S \vee U$ is homotopic to $\inc^{\hi} \col^{\hi} g + \inc^{\lo} \col^{\lo} g$.

On the first summand $S_1$, the map $f$ is
\begin{align*}
f \inc_1^{\hi} &= \inc_2^{\hi} \col_2^{\hi} f \inc_1^{\hi} + \inc_2^{\lo} \col_2^{\lo} f \inc_1^{\hi} \\
&= \inc_2^{\hi} f^{\hi} + \inc_2^{\lo} f^{\hi \lo}
\end{align*}
and on the second summand $U_1$, the map $f$ is
\begin{align*}
f \inc_1^{\lo} &= \inc_2^{\lo} \col_2^{\lo} f \inc_1^{\lo} \: \text{ (by cellular approximation)} \\
&= \inc_2^{\lo} f^{\lo}
\end{align*}
from which we obtain the desired equality $\psi \phy(f) = f$.

\textbf{Step 4: $\phy \psi = \id_{\TT}$.} On objects, the composite of functors does
\[
(S,U) \stackrel{\psi}{\mapsto} S \vee U \stackrel{\phy}{\mapsto} (S,U)
\]
and on a map $(f^h,f^l,\de) \colon (S_1,U_1) \to (S_2,U_2)$, the composite does
\begin{align*}
(f^h,f^l,\de) \stackrel{\psi}{\mapsto} &\left( \inc_2^{\hi} f^h + \inc_2^{\lo} \de \right) ; \inc_2^{\lo} f^l \\
\stackrel{\phy}{\mapsto} &\left( \col_2^{\hi} \left( \inc_2^{\hi} f^h + \inc_2^{\lo} \de \right), \col_2^{\lo} \inc_2^{\lo} f^l, \col_2^{\lo} \left( \inc_2^{\hi} f^h + \inc_2^{\lo} \de \right) \right) \\
= &\left( \col_2^{\hi} \inc_2^{\hi} f^h + \col_2^{\hi} \inc_2^{\lo} \de , \col_2^{\lo} \inc_2^{\lo} f^l, \col_2^{\lo} \inc_2^{\hi} f^h + \col_2^{\lo} \inc_2^{\lo} \de \right) \\
= &\left( f^h, f^l, \de \right). \qedhere
\end{align*}
\end{proof}

\begin{remark}
Proposition \ref{SplitLinExtn} was implicitly used in \cite[Proposition 1.6]{Baues00} without being proved there.
\end{remark}

\subsection*{Homotopy operation functors}

\begin{proof}[Proof of Proposition \ref{GammaAdditive}]
Let $A_n$ be an abelian group. We want to compute the abelian group $\tild{\Ga}_n^k(A_n) = \Ga_n^k(A_n, 0, \ldots, 0)$.

Our functor $\Ga_n^k$ is the functor denoted $\rho^* \De$ in \cite[(7.3)]{Baues99}. By Proposition \ref{SplitLinExtn} and \cite[Lemma 7.5; Lemma 7.10]{Baues99}, $\Ga_n^k$ can be computed using a free presentation, as we will explain shortly. Here we will implicitly use the identification $\model(\PI_{n+k}^0) \cong \Ab$ sending a model $M$ to the abelian group $M(S^{n+k})$.

Let $g \colon T \to S$ be a map between wedges of spheres of dimensions $n$, $n+1$, $\ldots$, $n+k-1$ satisfying
\begin{enumerate}
 \item $\coker \pi_n(g) = A_n$;
 \item $\coker \pi_i(g) = 0$ for $n < i < n+k$, that is, $\pi_i(g)$ is surjective in those degrees.
\end{enumerate}
Then the sequence of abelian groups
\begin{equation} \label{CokerColl}
\pi_{n+k}(T \vee S)_2 \ral{\pi_{n+k}(g,1)} \pi_{n+k}(S) \surj \tild{\Ga}_n^k(A_n) \to 0
\end{equation}
is exact, where the left-hand group is
\[
\pi_{n+k}(T \vee S)_2 \dfn \ker \left( \pi_{n+k}(T \vee S) \ral{\pi_{n+k}(0,1)} \pi_{n+k}(S) \right) 
\]
i.e. the kernel of the collapse map. In other words, our functor can be computed as $\tild{\Ga}_n^k(A_n) = \coker \pi_{n+k}(g,1)$.

A free presentation can be obtained as follows. Let $R \ral{f} F \surj A_n \to 0$ be a free presentation of $A_n$ as abelian group, i.e., an exact sequence where $R$ and $F$ are free abelian groups. Realize $R \to F$ as $\pi_n(g')$ for a map $g' \colon S' \to S$ between wedges of spheres of dimension $n$ (with a sphere $S^n$ for each summand $\Z$). Now insert spheres of higher dimensions to kill all the homotopy of $S$. More precisely, consider the wedge
\[
S'' \dfn \We_{\substack{x \in \pi_i S \\ n < i < n+k}} S^i
\]
and the map $g'' \colon S'' \to S$ defined on each summand $S^i$ by (a representative of) the indexing element $x \in \pi_i S$. The map
\[
T = S'' \vee S' \ral{g = (g'',g')} S
\]
provides a free presentation as described above.

\textbf{Step 1: Assume $A_n \simeq \Z$ is free on one generator.}

The free presentation of $A_n$ is given by $R=0$ and $F=\Z$, so that we take $S' = *$ and $S = S^n$. We want to compute the cokernel illustrated in \eqref{CokerColl}. We claim that the image of $\pi_{n+k}(g,1)$ is the subgroup $Dec \subset \pi_{n+k}(S^n)$ generated by decomposable elements, which would prove the result $\tild{\Ga}_n^k(\Z) = Q_{k,n}$.

Take $x \in \pi_{n+k}(T \vee S^n)_2$ and consider its image $\pi_{n+k}(g,1)(x) \in \pi_{n+k}(S^n)$ as illustrated in the diagram
\[
\xymatrix{
S^{n+k} \ar[dr] \ar[r]^{x} & T \vee S^n \ar[d]^{(g,1)} \\
& S^n. \\
}
\]
Since $T$ is a wedge of spheres (of dimensions strictly between $n$ and $n+k$), the Hilton--Milnor theorem \cite[Theorem XI.8.1]{Whitehead78} implies
\[
\pi_{n+k}(T \vee S^n) \simeq \Op_j \pi_{n+k} (S^{m_j})
\]
for some appropriate dimensions $m_j$, and $x$ can be expressed as
\[
x = \sum_j p_j \circ x_j
\]
where the $p_j$ are certain iterated Whitehead products of summand inclusions of the individual spheres of $T \vee S^n$. In particular, the element
\[
(g,1) \circ x = (g,1) \circ \left( \sum_j p_j \circ x_j \right) = \sum_j (g,1) \circ p_j \circ x_j
\]
is a sum of decomposables, except possibly one term, corresponding to the summand inclusion $S^n \inj T \vee S^n$. However, that one term is precisely $x_j = (0,1) \circ x = \pi_{n+k}(0,1)(x) = 0$ by assumption. Hence $\pi_{n+k}(g,1)(x)$ is decomposable.

Conversely, take any decomposable element $x \in \pi_{n+k}(S^n)$. By the assumption $k \neq n-1$, $x$ must be a sum of compositions $x = \sum_i x_i \circ \al_i$ for some $\al_i \in \pi_{n+k}(S^{m_i})$, $x_i \in \pi_{m_i}(S^n)$, $n < m_i < n+k$. But each such composite is in the image of $\pi_{n+k}(g,1)$. By construction of $T$, there is a wedge summand $S^{m_i} \inj T$ corresponding to $x_i \in \pi_{m_i}(S^n)$. The diagram
\[
\xymatrix{
S^{n+k} \ar[drr] \ar[r]^{\al_i} & S^{m_i} \ar[dr]^{x_i} \ar@{^{(}->}[r]^{\io} & T \vee S^n \ar[d]^{(g,1)} \\
& & S^n. \\
}
\]
illustrates the equality $x_i \circ \al_i = (g,1) \circ \io \circ \al_i = \pi_{n+k}(g,1) (\io \circ \al_i)$. Moreover, the map $(0,1) \circ \io \colon S^{m_i} \to S^n$ is null, which guarantees $\io \circ \al_i \in \ker \pi_{n+k}(0,1) = \pi_{n+k}(T \vee S^n)_2$.

\textbf{Step 2: Assume $A_n$ is free.}

Take $S = \vee_l S^n$ satisfying $A_n = F \simeq \op_l \Z = \pi_n (S)$ and take $S' = *$. Consider the composition function
\begin{align*}
\pi_n (S) \x \pi_{n+k}(S^n) &\to \pi_{n+k}(S) \\
(x,\al) &\mapsto x \circ \al.
\end{align*}
It is linear in the second variable $\al$ but not in the first variable $x$. Failure to be linear in $x$ is measured by the ``distributive law of homotopy theory'' or Hilton's formula \cite[Theorem XI.8.5]{Whitehead78}. The error terms are composites which are all in the image of $\pi_{n+k}(g,1) \colon \pi_{n+k}(T \vee S)_2 \to \pi_{n+k}(S)$ as explained in step 1. By modding out this image, we obtain a well-defined bilinear map 
\[
\pi_n(S) \ot \pi_{n+k}(S^n) \to \tild{\Ga}_n^k (A_n).
\]
This map vanishes on elements $x \ot \al$ where $\al$ is decomposable, since such an $\al$ is in the image of $\pi_{n+k}(g,1)$. Thus there is an induced canonical map
\[
\phy \colon \pi_n(S) \ot Q_{k,n} \to \tild{\Ga}_n^k (A_n).
\]
We claim that $\phy$ is an isomorphism. The Hilton--Milnor theorem provides an isomorphism
\begin{align*}
\pi_{n+k}(S) &= \pi_{n+k}(\vee_l S^n) \\
&\simeq \Op_j \pi_{n+k}(S^{m_j}) \\
&\simeq \Op_l \pi_{n+k}(S^n) \op \Op_{j \text{ such that } m_j > n} \pi_{n+k}(S^{m_j})
\end{align*}
so that we can project onto the first summand $\op_l \pi_{n+k}(S^n) \cong F \ot \pi_{n+k}(S^n)$ and then mod out the decomposables:
\[
\pi_{n+k}(S) \surj F \ot \pi_{n+k}(S^n) \surj F \ot Q_{k,n} = \pi_n(S) \ot Q_{k,n}.
\]
This map vanishes on the image of $\pi_{n+k}(g,1)$ and therefore induces a map on the cokernel
\[
\psi \colon \tild{\Ga}_n^k (A_n) \to \pi_n(S) \ot Q_{k,n}.
\]
One readily checks that $\psi$ is inverse to $\phy$.

\textbf{Step 3: $A_n$ is an arbitrary abelian group.}

The free presentation of $A_n$ can be canonically turned into the reflexive coequalizer diagram:
\[
\xymatrix{
R \op F \ar@<+0.8ex>[r]^-{(f,1)} \ar@<-0.8ex>[r]_-{(0,1)} & F \ar[l] \ar@{->>}[r] & A_n \\
}
\]
where the summand inclusion $F \inj R \op F$ is a common section of the pair of maps. Since the functor $- \ot Q_{k,n} \colon \Ab \to \Ab$ preserves reflexive coequalizers (in fact it is additive and right exact), it suffices to show that $\tild{\Ga}_n^k$ preserves reflexive coequalizers to obtain the natural isomorphism
\[
\tild{\Ga}_n^k (A_n) = A_n \ot Q_{k,n}
\]
using Step 2.

To prove that $\tild{\Ga}_n^k$ preserves reflexive coequalizers, recall that this functor is the composite
\[
\xymatrix{
\Ab \cong \PiAlg_n^0 \ar@/_2pc/[rrr]_-{\tilde{\Ga}_n^k} \ar@{^{(}->}[r]^-{\io} & \PiAlg_n^{k-1}  \ar@/^2pc/[rr]^-{\Ga_n^k} \ar[r]^-L & \PiAlg_n^{k} \ar[r]^-{\pi_{n+k}} & \Ab \\
}
\]
where $L$ is left adjoint to Postnikov truncation, and in particular $L$ preserves co\-li\-mits. The inclusion $\io \colon \PiAlg_n^0 \to  \PiAlg_n^{k-1}$ admits a right adjoint, and thus preserves colimits. By \cite[Chapter 3]{Adamek11}, reflexive coequalizers in $\PiAlg_n^{k}$ are computed at the level of underlying graded sets, and are in particular preserved by the restriction functor $\pi_{n+k} \colon \PiAlg_n^{k} \to \Ab$.
\end{proof}

\begin{proof}[Proof of Proposition \ref{GammaQuadratic}]
Similar to the proof of \ref{GammaAdditive} above. The key ingredient here is the computation of \cite[Corollary 9.4]{Baues94}:
\[
\pi_{2n-1}(S) \cong \pi_n(S) \ot^q \pi_{2n-1} \{ S^n \}
\]
where $S = \vee_l S^n$ is a wedge of $n$--spheres, so that $\pi_n(S) \cong \op_l \Z$ is a free abelian group. Decomposables (compositions) must be modded out for the same reason as in the proof of \ref{GammaAdditive}.

The functor $- \ot^q Q_{n-1} \{ S^n \} \colon \Ab \to \Ab$ is not additive and does not preserve cokernels in general, but it does preserve reflexive coequalizers.
\end{proof}

\bibliographystyle{plain}
\bibliography{RealizabilityOps_Bib}   

\end{document}